\documentclass[usletter, 10pt, conference]{ieeeconf}
\IEEEoverridecommandlockouts
\usepackage{amssymb}
\usepackage{amsmath}
\usepackage{graphicx}
\usepackage{amsfonts}
\usepackage{cite}

\usepackage{algpseudocode}
\usepackage{algorithm}

\usepackage{dsfont,accents}
\usepackage{url}
\usepackage{hyperref}
\hypersetup{
    colorlinks,
    citecolor=black,
    filecolor=black,
    linkcolor=black,
    urlcolor=black,
    bookmarks=true,
    bookmarksnumbered=true,
    bookmarksdepth=4
}

\renewcommand{\theenumi}{\alph{enumi}}

\newtheorem{theorem}{Theorem}

\newtheorem{proposition}[theorem]{Proposition}

\newtheorem{hyp}[theorem]{Assumption}
\newtheorem{remark}[theorem]{Remark}

\newtheorem{problem}[theorem]{Problem}

\newcommand{\mendth}{\hfill \ensuremath{\vartriangle}}

\DeclareMathOperator*{\col}{col}

\def\E{\mathbb{E}}

\def\d{\mathrm{d}}

\makeatletter
\newcommand{\pushright}[1]{\ifmeasuring@#1\else\omit\hfill$\displaystyle#1$\fi\ignorespaces}
\newcommand{\pushleft}[1]{\ifmeasuring@#1\else\omit$\displaystyle#1$\hfill\fi\ignorespaces}
\makeatother

\title{Integral population control of a quadratic dimerization process}

\author{Corentin Briat and Mustafa Khammash
\thanks{Corentin Briat and Mustafa Khammash are with the Department of Biosystems Science and Engineering {(D-BSSE)}, Swiss Federal Institute of Technology--Z\"{u}rich {(ETH--Z)}, Mattenstrasse 26, 4058 Basel, Switzerland; email: {\tt  \{corentin.briat,mustafa.khammash\}@bsse.ethz.ch}; url: \protect\url{http://www.bsse.ethz.ch/ctsb}, \protect\url{http://www.briat.info}, \protect\url{http://www.bsse.ethz.ch/research/Professors/khammash\textunderscore cv}}}

\begin{document}
\maketitle
\vspace{-2cm}
\begin{abstract}
Moment control of a simple quadratic reaction network describing a dimerization process is addressed. It is shown that the moment closure problem can be circumvented without invoking any moment closure technique. Local stabilization and convergence of the average dimer population to any desired reference value is ensured using a pure integral control law. Explicit bounds on the controller gain are provided and shown to be valid for any reference value. As a byproduct, an explicit upper-bound of the variance of the monomer species, acting on the system as unknown input due to the moment openness, is obtained. The obtained results are illustrated by an example relying on the simulation of a cell population using stochastic simulation algorithms.
\end{abstract}


\section{Introduction}

Stochastic reaction networks are important modeling tools used in biology to mathematically represent, among others, chemical reaction networks where different interacting species, such as genes, mRNAs or proteins, are involved. Deterministic reaction networks \cite{Feinberg:72,Horn:72} have been extensively considered to model such networks until it has been more recently pointed out that randomness plays a dominant role when species are in low copy numbers, which is quite often the case in biological systems. It has been indeed noticed in \cite{Elowitz:02} that identical cells may exhibit a dramatically different quantitative behavior, emphasizing then the preponderance of random effects, or \emph{intrinsic noise} \cite{Elowitz:02}. Several studies showed that noise plays an important role by allowing life to achieve certain functions that would be difficult or impossible to realize in a deterministic setting; see for instance the stochastic circadian clock model of \cite{Barkai:99} or the Pap pili epigenetic switch of \cite{Jarboe:04}. On the other hand, noise induces variability in population levels and makes regulation tasks more difficult to achieve. But, life evolved and optimized regulation circuits to make them robust against noise.

It has been shown in \cite{Khammash:11} that, in a population of genetically engineered light-responding yeast cells, it was possible to control the average protein levels in cell populations by externally acting on gene expression rates using light. Based on a dynamical model describing the average populations of mRNA and protein molecules of a gene expression network, the control problem has been solved using Kalman filtering and model-predictive control. This approach has then been extended to mean and variance control in \cite{Briat:12c} using elementary integral control laws. It is notably shown there that fundamental limits, in terms of achievable mean and variance, are imposed by the network topology and cannot be overstepped by any computer-based control technique. Single cell control and population control have also been implemented using microfluidics and model predictive control in \cite{Uhlendorf:12}.

The above approaches have to be contrasted with control law implementations relying on synthetic biology where controllers are implemented inside cells, therefore consisting themselves of biologically interacting species, see e.g. \cite{Klavins:10,Oishi:10b}. Inner and outer control approaches are fully complementary since synthetic control networks are limited in terms of  flexibility and computational power, whereas outer control is not. Synthetic controllers are, on the other hand, able to consistently modify the structure of the controlled network which may enhance its noise reduction properties; e.g. through variance reduction.

The goal of this paper is to go beyond stochastic reaction networks with affine propensities for population control \cite{Khammash:11,Briat:12c}, and emphasize that a more general theory encompassing a wider class of networks, notably with quadratic propensities, might be possible to develop using moments equations. It is indeed well known that when quadratic propensities are involved, the moment closure problem arises \cite{Singh:11,Gillepsie:09,Milner:11}, and seems to compromise then the entire framework based on moment equations.

In the current paper, we consider a very simple instance of a quadratic network that is simple enough to obtain analytical results, but complicated enough to present all the characteristics of more general quadratic reaction networks, i.e. unknown input signals correlated to the state and quadratic nonlinearities. We first show that the considered reaction network is structurally exponentially ergodic, which implies that, for any values of the parameters, there exists a unique attractive stationary distribution. We, moreover, prove that all the moments exist and are globally exponentially converging to their unique equilibrium value. These results motivate the use of a simple integral control law since stabilization of the process is not necessary, only reference tracking is demanded. The main result of the paper, addressing local asymptotic stabilization of the controlled network around any suitable, but partially unknown, equilibrium point using integral control is provided next. Explicit bounds on the controller gain are derived in terms of the network parameters and shown to be independent of the equilibrium point, making therefore the control law generic for this type of network. An analytical bound on the variance is also obtained in the process. The theoretical results are then illustrated through an example relying on stochastic simulation algorithms simulating a cell population.

\textit{Outline:} The structure of the paper is as follows. The problem is stated in Section \ref{sec:prob} and the main results are obtained in Section \ref{sec:results}. An illustrative example is finally discussed in Section \ref{sec:ex}.


\section{Problem statement}\label{sec:prob}

\subsection{General framework}

Assume $N$ molecular species $S_1,\ldots,S_N$ interacting with each others through $M$ reaction channels $R_1,\ldots,R_M$. Under the assumption of homogeneous mixing and thermal equilibrium, the time evolution of the random variables $X_1(t),\ldots,X_N(t)$ associated with the population of each species can be described by the so-called Chemical Master Equation (CME), or Forward Kolmogorov equation, given by
\begin{equation*}
  \dot{P}(\varkappa,t)=\sum_{k=1}^M\left[w_k(\varkappa-s_k)P(\varkappa-s_k,t)-w_k(\varkappa)P(\varkappa,t)\right]
\end{equation*}
where $s_k$ is the stoichiometry vector associated with reaction $R_k$ and $w_k$ the propensity function capturing the rate of the reaction $R_k$. The variable $\varkappa$ is the state-variable and $P(\varkappa,t)$ denotes the probability to be in state $\varkappa\in\mathbb{Z}_{\ge0}^N$ at time $t$.

Based on the CME, the following dynamical model for the first-order moments can be easily obtained
\begin{equation}\label{eq:moments}
    \dfrac{\d \E[X]}{\d t}=S\E[w(X)]
\end{equation}
where $S:=\begin{bmatrix}
  s_1 & \ldots & s_M
\end{bmatrix}\in\mathbb{R}^{N\times M}$ is the stoichiometry matrix and $w(X):=\begin{bmatrix}
  w_1(X)^T & \ldots & w_M(X)^T
\end{bmatrix}^T\in\mathbb{R}^{M}$ the propensity vector.

Whenever the propensity functions are affine, the above dynamical model is well-defined in the sense that the moment trajectories are uniquely defined by initial conditions $\E[X(0)]$. When the propensity functions are nonlinear, we face the moment closure problem corresponding to the fact that moment dynamics depend on moments of nonlinear functions of the random variable $X(t)$. When, for instance, propensities are quadratic, the first-order moments depend on the second-order moments, and so forth. We therefore end up, in the latter case, with an infinite number of linear differential equations. If we, however, restrict ourselves to the dynamics of the first order moments, the resulting system of differential equations will be open, i.e. will have inputs that are, somehow, correlated to the state. In such a case, initial conditions are not sufficient anymore for fully defining a trajectory solution for \eqref{eq:moments} since we also need the values of the inputs at any time, which are most of time not directly computable. A way for resolving this problem consists of \emph{closing the moments} by, for instance, expressing the inputs as functions of the state of the system or by neglecting higher-order cumulants; see e.g. \cite{Gomez:07,Gillepsie:09,Singh:11,Milner:11}. We shall, however, not use any closure technique in the current paper and attack the problem directly.

\subsection{A quadratic dimerization process}

Mean control and mean/variance control of a gene-expression network, which is an affine network, have been performed in \cite{Khammash:11} and \cite{Briat:12c}, respectively. The goal here is to go beyond affine networks and show that similar ideas can still be applied, even in presence of closedness problems. We will therefore focus on the following stochastic chemical reaction network
\begin{equation}\label{eq:reacnet}
\begin{array}{lccclclcccl}
   R_1&:&\phi&\stackrel{k_1}{\longrightarrow}&S_1,&&  R_2&:&S_1+S_1&\stackrel{b}{\longrightarrow}&S_2,\\
   R_3&:&S_1&\stackrel{\gamma_1}{\longrightarrow}&\phi,&&  R_4&:&S_2&\stackrel{\gamma_2}{\longrightarrow}&\phi
 \end{array}
\end{equation}
in which the protein $S_1$ dimerizes into $S_2$ at rate $b$. As it will be explained later, this network is simple enough to obtain analytical results, but complicated enough to exhibit all the difficulties arising in the control of stochastic quadratic reaction networks. The goal of the paper is to provide a solution to the following problem:
\begin{problem}
  Design a controller such that the average dimer population $\E[X_2(t)]$ locally exponentially converges to the reference $\mu$.
\end{problem}

The first step towards a suitable solution of the problem above, consists of defining a model for the dynamics of the first-order moments \cite{Khammash:11,Briat:12c}. For this specific network, we indeed have:
\begin{proposition}
The first-order moment dynamics are described by the open system of nonlinear differential equations
\begin{equation}\label{eq:syst}
  \begin{array}{lcl}
    \dot{x}_1(t)&=&k_1+(b-\gamma_1)x_1(t)-bx_1(t)^2-bv(t)\\
    \dot{x}_2(t)&=&-\dfrac{b}{2}x_1(t)-\gamma_2 x_2(t)+\dfrac{b}{2}x_1(t)^2+\dfrac{b}{2}v(t)
  \end{array}
\end{equation}
where $x_i(t):=\E[X_i(t)]$, $i=1,2$, and $v(t):=V(X_1(t))$ is the variance of the random variable $X_1(t)$.\mendth
\end{proposition}
\begin{proof}
  The proof follows from the application of the general formula \eqref{eq:moments} with
  \begin{equation*}
    S=\begin{bmatrix}
      1 & -2 & -1 & 0\\
      0 & 1 & 0 & -1
    \end{bmatrix}
  \end{equation*}
  and
  $ w(X)=\begin{bmatrix}
      k_1 & \frac{b}{2}X_1(X_1-1) & \gamma_1X_1 & \gamma_2X_2
    \end{bmatrix}^T$.
  Using finally the identity $V(X_1)=\E[X_1^2]-\E[X_1]^2$, the result is obtained.
\end{proof}

\subsection{Main difficulties}\label{sec:diff}

In spite of being simple, the  network \eqref{eq:reacnet} presents all the difficulties that can arise in quadratic reaction networks and is a good candidate for emphasizing that moment control can be analytically solvable, even in presence of the moment closure problem. Below is a list of difficulties that are specific to network \eqref{eq:reacnet} and, a fortiori, specific to any network having quadratic reactions:
\begin{enumerate}
  \item The system \eqref{eq:syst} has the variance $v(t):=V(X_1(t))$ as input signal and it is not known, a priori, whether it is bounded over time or even asymptotically converging to a finite value $v^*$.
  \item The system \eqref{eq:syst} is nonlinear and nonlinear terms may not be neglected since they may enhance certain properties such as stability. It will be shown later that this is actually the case for system \eqref{eq:syst}.
  \item Due to our complete ignorance in the value of $v^*$ (if it exists), the system \eqref{eq:syst} exhibits an infinite number of equilibrium points. Understand this, however, as an artefact arising from the definition of the model \eqref{eq:syst} since the first-order moments may, in fact, have a unique stationary value.
\end{enumerate}

\section{Main results}\label{sec:results}

\subsection{Preliminaries}

The following result proves a crucial stability property for our process:
\begin{theorem}\label{th:ergodicd}
  For any value of the network parameters $k_1,b,\gamma_1$ and $\gamma_2$, the reaction network \eqref{eq:reacnet} is exponentially ergodic and has all its moments bounded and globally exponentially converging. Notably, for any initial state $X(0)$ of the Markov process, there exists a unique $v^*\ge0$ such that $v(t)\to v^*$ as $t\to\infty$.\mendth
  %
\end{theorem}
\begin{proof}
See Appendix \ref{ap:ergodic}.
\end{proof}

The above result provides an answer to the first difficulty mentioned in Section \ref{sec:diff}. It indeed states that, for any parameter configuration, all the moments are bounded and exponentially converging to a unique stationary value. 

The next step consists of choosing a suitable control input, that is, a control input from which any reference value $\mu$ for $x_2$ can be tracked. We propose to use the production rate $k_1$ as control input. To prove that this control input is judicious we need the following assumption motivated by the structure of the network \eqref{eq:reacnet}:
%
%
%
\begin{hyp}\label{ass:2}
The function $S^*:=x_1^{*2}-x_1^*+v^*$, where $x_1^*$ is the equilibrium solution for $x_1$ and $v^*$ is the equilibrium variance, verifying the equation
  \begin{equation}
    k_1-\gamma x_1^*-bS^*=0,
  \end{equation}
is a continuous function of $k_1$.\mendth
\end{hyp}

By indeed increasing $k_1$, we will have more $X_1$ at stationarity, and consequently more $X_1^2$. It seems important to stress here that the continuity of the stationary distribution with respect to the network parameters cannot be assessed from the continuity of the probability distribution over time since the limit of continuous functions need not be continuous. Therefore, an argument based on the continuity of the stationary distribution seems difficult to consider.

Based on the above assumption, we can state the following result:
\begin{proposition}\label{prop:mu}
 For any $\mu>0$, there exists $k_1(\mu)>0$ such that we have $x_2^*=\mu$ where $x_2^*$ is the unique stationary value for $\E[X_2]$.\mendth
\end{proposition}
\begin{proof}
  See Appendix \ref{ap:mu}.
\end{proof}
From the results stated in Theorem \ref{th:ergodicd} and Proposition \ref{prop:mu}, it seems reasonable to consider a pure integral control law since exponential stability nominally holds and only tracking is necessary. Therefore, we propose that $k_1$ be actuated as
\begin{equation}\label{eq:cl}
\begin{array}{rcl}
  \dot{I}(t)&=&\mu-x_2(t)\\
    k_1(t)&=&k_c\varphi(I(t))
\end{array}
\end{equation}
where $k_c>0$ is the gain of the controller, $\mu$ is the reference to track and $\varphi(y):=\max\{0,y\}$. 


\subsection{Nominal stabilization result}

We are now in position to state the main result of the paper:
\begin{theorem}[Main stabilization result]\label{th:main}
  For any finite positive constants $\gamma_1,\gamma_2,b,\mu$ and any controller gain $k_c$ satisfying
  \begin{equation}\label{eq:kcbound}
  0<k_c<2\gamma_2\left(2\gamma_1+\gamma_2+2\sqrt{\gamma_1(\gamma_1+\gamma_2)}\right),
\end{equation}
  the closed-loop system \eqref{eq:syst}-\eqref{eq:cl} has a unique locally exponentially stable equilibrium point $(x_1^*,x_2^*,I^*)$ in the positive orthant such that $x_2^*=\mu$.
The equilibrium variance moreover satisfies
\begin{equation*}
\begin{array}{lcr}
   \quad\quad\quad\quad\quad\quad&v^*\in\left(0,\dfrac{2\gamma_2\mu}{b}+\dfrac{1}{4}\right].& \quad\quad\quad\quad\quad\quad\vartriangle
\end{array}
\end{equation*}
\end{theorem}
\begin{proof}
The proof is given in the Appendix \ref{ap:main}.
\end{proof}

The above result states two important facts that must be emphasized. First of all, the condition on the controller gain is uniform over $\mu>0$ and $b>0$, and is therefore valid for any combination of these parameters. This also means that a single controller, which locally stabilizes all the possible equilibrium points, is easy to design for this network. Second, the proof of the theorem provides an explicit construction of an upper-bound on the equilibrium variance $v^*$, which turns out to be a linearly increasing function of $\mu$. This upper-bound is, moreover, tight when regarded as a condition on the equilibrium points of the system since, when the equilibrium variance $v^*$ is greater than $2\gamma_2\mu/b+1/4$, the system does not admit any real equilibrium point.

\subsection{Robust stabilization result}

Let us consider the following set
\begin{equation}
\mathcal{P}:=[\gamma_1^-,\gamma_1^+]\times[\gamma_2^-,\gamma_2^+]\times[b^-,b^+]
\end{equation}
defined for some appropriate positive real numbers $\gamma_1^-<\gamma_1^+$, $\gamma_2^-<\gamma_2^+$ and $b^-<b^+$. We get the following generalization of Theorem \ref{th:main}:
\begin{theorem}[Robust stabilization result]
Assume the controller gain $k_c$ verifies
  \begin{equation}
0<k_c<2\gamma_2^-\left(2\gamma_1^-+\gamma_2^-+2\sqrt{\gamma_1^-(\gamma_1^-+\gamma_2^-)}\right).
\end{equation}
Then, for all $(\gamma_1,\gamma_2,b)\in\mathcal{P}$, the closed-loop system \eqref{eq:syst}-\eqref{eq:cl} has a unique locally stable equilibrium point $(x_1^*,x_2^*,I^*)$ in the positive orthant such that $x_2^*=\mu$. The equilibrium variance $v^*$, moreover, satisfies
\begin{equation*}
\begin{array}{lcr}
  \quad\quad\quad\quad\quad\quad& v^*\in\left(0,\dfrac{2\gamma_2^+\mu}{b^-}+\dfrac{1}{4}\right]. & \quad\quad\quad\quad\quad\quad\vartriangle
\end{array}
\end{equation*}
\end{theorem}
\begin{proof}
  The upper bound on the controller gain is a strictly increasing function of $\gamma_1$ and $\gamma_2$, and the most constraining value (smallest) is therefore attained at $\gamma_1=\gamma_1^-$ and $\gamma_2=\gamma_2^-$. A similar argument is applied to the variance upper-bound.
\end{proof}

\subsection{Additional remarks}

The following remark addresses the point that $\E[X_1(t)^2]$, and hence nonlinearities, cannot be neglected in the current problem since, without them, stabilization using the control-law \ref{eq:cl} is not even possible:
\begin{remark}
If we were, indeed, restricting ourselves to the simplified homogeneous linear dynamics
  \begin{equation}\label{eq:systL}
\dot{y}(t)=\begin{bmatrix}
  b-\gamma_1 & 0 & k_c\\
  -\frac{b}{2} & -\gamma_2 & 0\\
  0 & -1 & 0
\end{bmatrix}y(t)+\begin{bmatrix}
  0\\
  0\\
  \mu
\end{bmatrix},
\end{equation}
we would incorrectly conclude that  1) the uncontrolled system may be unstable since the $2\times 2$ left-upper block is not Hurwitz whenever $b-\gamma_1>0$; and that 2) the system cannot be controlled by an integrator since the determinant of the system matrix is given by $\frac{bk_c}{2}>0$, implying then that the closed-loop system matrix is not Hurwitz\footnote{A necessary condition for a $3\times 3$ matrix to be Hurwitz is negativity of the determinant.}.\mendth
\end{remark}


The following remark addresses a key point in the linearization procedure of the moment equations:
\begin{remark}
  The linearized systems used for proving Theorem \ref{th:main} involve local variations of the variance as inputs. Only the variance equilibrium value $v^*$ has impact on local stability. We may ask whether this is technically correct. The main difficulty here lies in the fact that $v(t)$ can be very complicated and does not necessarily depends explicitly on $I(t)$, even if its equilibrium value $v^*$ does depend on $I^*$. If we assume independence of $I(t)$ and $v(t)$, the presented results are valid. If we assume, however, that $v(t)$ is a differentiable and increasing function of $I(t)$ (at least very locally), then the conclusions of Theorem \ref{th:main} are still valid. This can be proved by writing the local linear systems and looking at the Routh-Hurwitz conditions. In this case, the newly introduced terms depending on $dv/dI$ will actually improve stability properties of the system by  enlarging the admissible controller parameter space. The worst-case (most constraining) scenario is, interestingly, when $dv/dI=0$, that is when $v(t)$ is independent of $I(t)$ and, in this case, the results of Theorem \ref{th:main} are retrieved.\mendth
\end{remark}

\section{Example}\label{sec:ex}

\begin{algorithm*}
  \caption{Algorithm for simulating the controlled cell population}

  \begin{algorithmic}[1]
   \Require{$T_s,\mu,k_c,T>0$, $N,N_p\in\mathbb{N}$, $\{x_0^1,\ldots,x_0^N\}\in\left(\mathbb{N}_0^2\right)^N$, $I_0\in\mathbb{R}$ and $p\in\mathbb{R}_{>0}^{N_p}$}
   \State Create array $t$ of time instants from 0 to $T$ with time-step $T_s$.
   \State $N_s=\text{length}(t)$
   \State Initialize: $i\leftarrow1$, $y\leftarrow\text{mean}(x_0)$, $I\leftarrow I_0$
    \For{$i<N_s$}
    \State Update control input: $u\leftarrow k_c\cdot\max\{0,I\}$
    \State Update controller state: $I\leftarrow I+T_s(\mu-y)$
    \State Simulation of $N$ cells from time $t[i]$ to $t[i+1]$ with control input $u$ and network parameters $p$
    \State Update output: $y\leftarrow\text{mean}(protein\ population)$
    \State $i\leftarrow i+1$
    \EndFor
  \end{algorithmic}
\end{algorithm*}

Let us consider in this section the stochastic reaction network \eqref{eq:reacnet} with parameters $b=3$, $\gamma_1=2$ and $\gamma_2=1$. From condition \eqref{eq:kcbound}, we get that $k_c$ must satisfy $  0<k_c< 19.798$
to have local stability of the unique equilibrium point in the positive orthant. We then run Algorithm 1 with the controller gain $k_c=1$, a sampling period of $T_s=10$ms, the reference $\mu=5$, the controller initial condition $I(0)=0$, a population of $N=2000$ cells and initial conditions $x_0^i$ randomized in $\{0,1\}^2$, $i=1,\ldots,N$. The simulation results are depicted in Fig. \ref{fig:cell} to \ref{fig:variance}. We can clearly see in Fig.~\ref{fig:cell} that $\E[X_2(t)]$ tracks the reference $\mu$ reasonably well. The variance of $X_1(t)$, plotted in Fig. \ref{fig:variance}, is also verified to lie within the theoretically determined range of values. We indeed have $V(X_1(t))\simeq1.5$ in the stationary regime whereas the upper-bound is equal to $3+7/12\simeq3.583$. Moreover, since the variance at equilibrium is smaller than $\frac{2\gamma_2\mu}{b}=10/3$, we then have $v-2\gamma_2\mu/b<0$ and therefore case 1) holds in the proof of Theorem \ref{th:main}. It is, however, unclear whether this is also the case for any combination of network parameters.

\begin{figure}
\centering
  \includegraphics[width=0.3\textwidth]{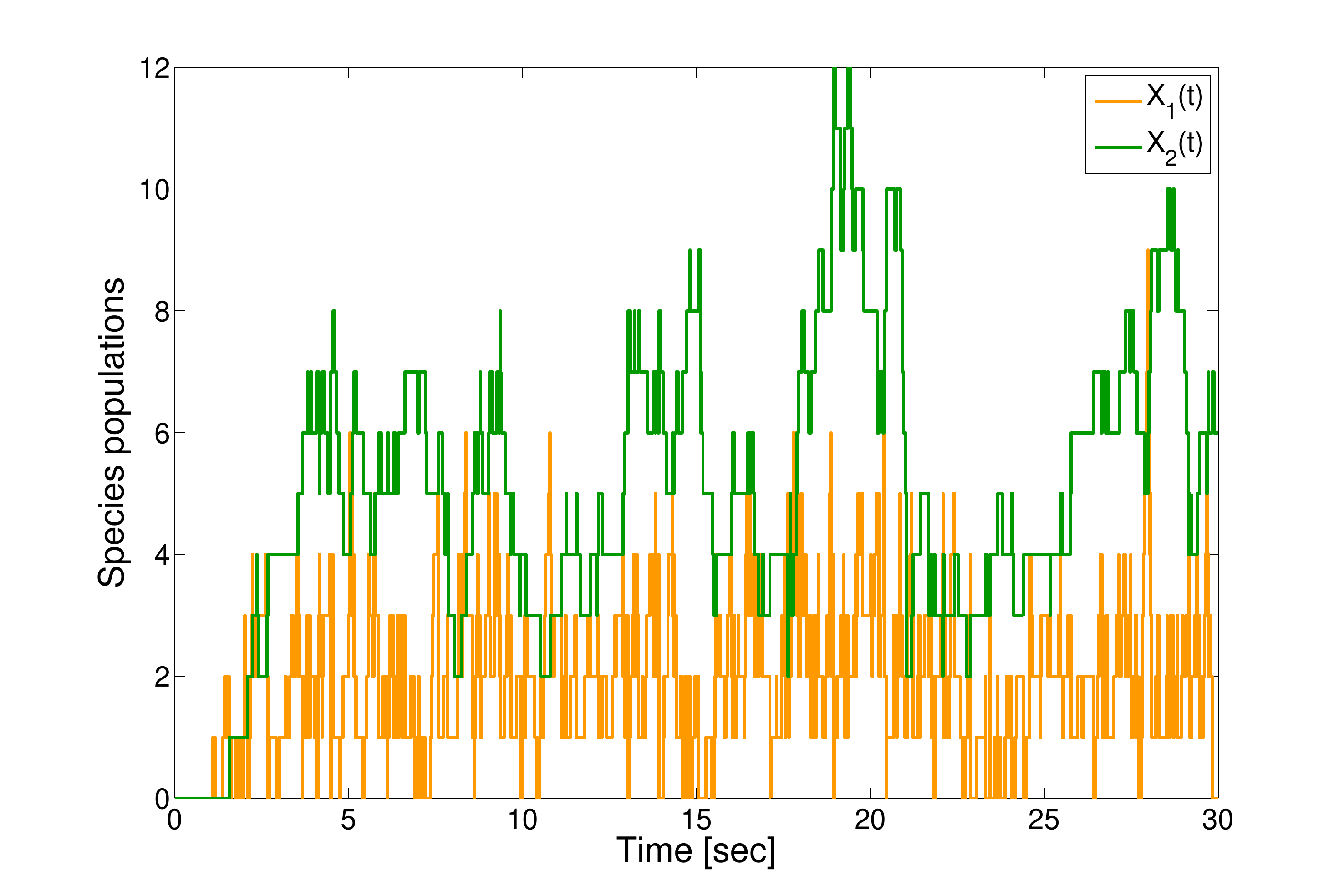}
  \caption{Evolution of proteins populations in a single cell}\label{fig:cell}
\end{figure}

\begin{figure}
\centering
  \includegraphics[width=0.3\textwidth]{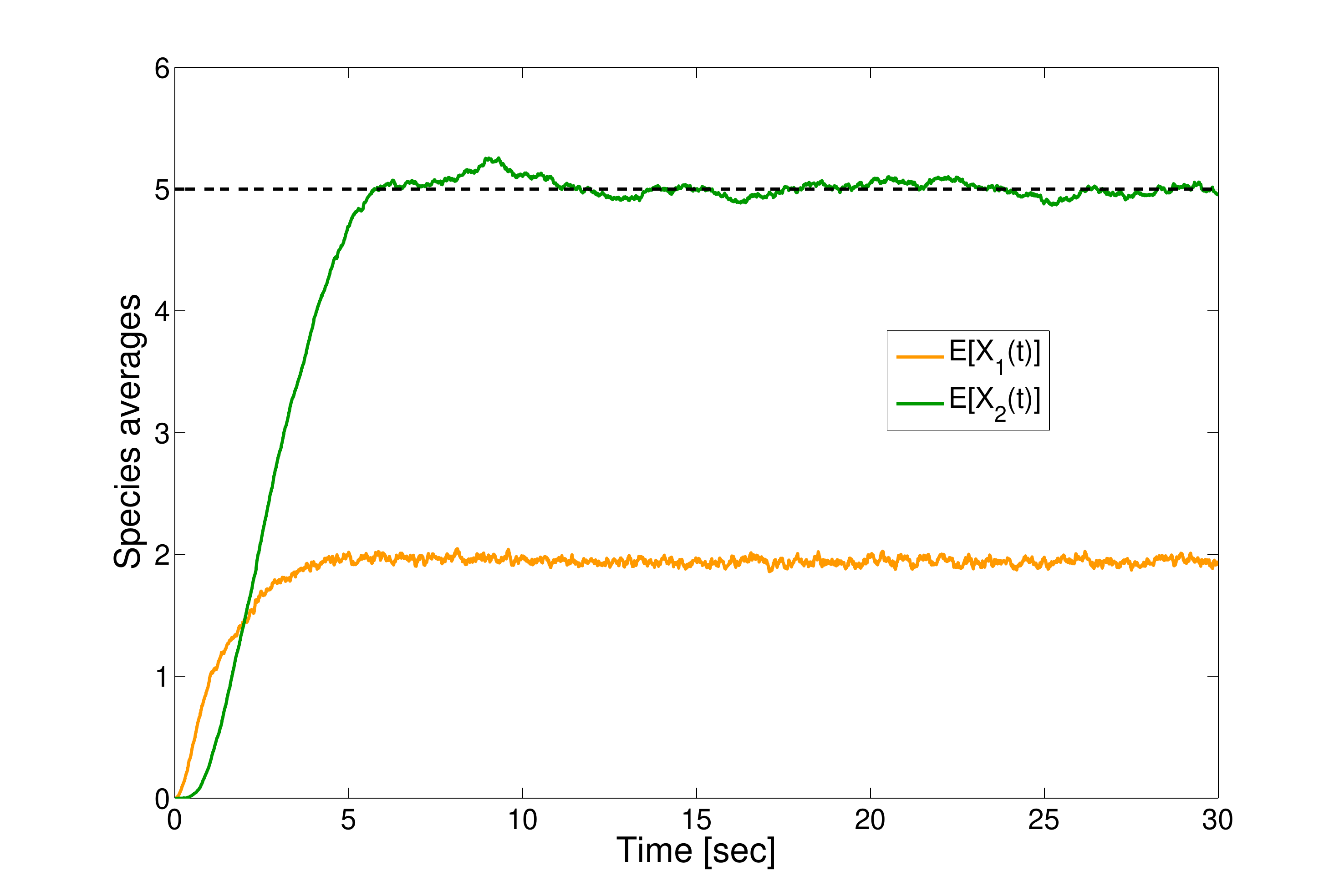}
  \caption{Evolution of the proteins averages in a population of 2000 cells}\label{fig:average}
\end{figure}


\begin{figure}
\centering
  \includegraphics[width=0.3\textwidth]{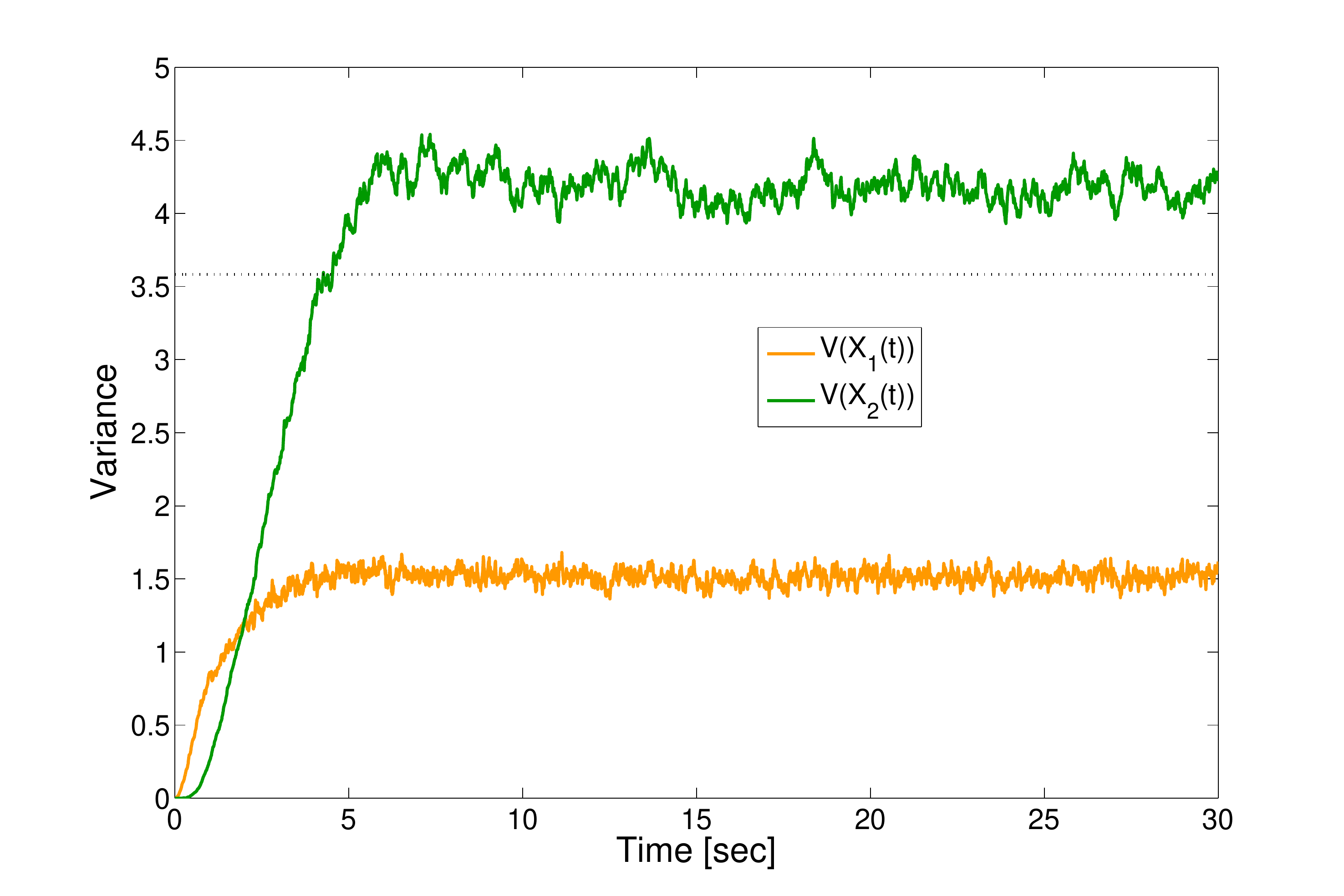}
  \caption{Evolution of the variances (computed with a population of 2000 cells). The dashed-line corresponds to the upper-bound on the equilibrium variance $V(X_1)$.}\label{fig:variance}
\end{figure}

\section{Conclusion}

The control problem of a the average population of dimers in a cell population has been proved to be solvable using integral feedback. Uniform bounds for the controller gains guaranteeing local asymptotic stability of a unique positive equilibrium point have been obtained. Interestingly, these bounds do not depend on the reference $\mu$ and the binding rate $b$, and are therefore valid for any reference value and a wide family of networks. One important emphasis of the proposed methodology is that the moment closure problem may not be a critical problem for control design.

\appendix

%

\subsection{Proof of Theorem \ref{th:ergodicd}}\label{ap:ergodic}

The proof is based on the results of \cite{Briat:13g}. Let us consider the reaction network \eqref{eq:reacnet}. It is easily seen that the network is irreducible since any state can be reached from any state using a sequence of reactions having positive propensities. We now recall a result from \cite{Briat:13g} (adapted to the current setup):
  \begin{theorem}\label{th:ergodic}
    Let the function $V(x):=\nu^T x$ where $\nu\in\mathbb{R}^2_{>0}$. Assume the considered reaction network is irreducible and that there exist positive constants $c_1,c_2,c_3,c_4$ and a vector $\nu\in\mathbb{R}^2_{>0}$ such that the conditions
    \begin{subequations}\label{eq:dd}
  \begin{align}
    LV(x)&\le c_1-c_2V(x)\label{eq:dd1} \\
   LV(x)^2-(LV(x))^2&\le c_3+c_4V(x) \label{eq:dd2}
  \end{align}
\end{subequations}
    for all $x\in\mathbb{Z}_{\ge0}^2$ where $L$ is the generator of the Markov process corresponding to the reaction network.

    Then, the Markov process is ergodic and has all its moments bounded and exponentially converging.\mendth
  \end{theorem}

   Considering then the inequality \eqref{eq:dd1} we have, for the reaction network \eqref{eq:reacnet}, that
  \begin{equation}
  \begin{array}{lcl}
    LV(x)   &=&     k_1\nu_1+\dfrac{b}{2}x_1(x_1-1)(\nu_2-2\nu_1)-\gamma_1x_1\nu_1\\
    &&-\gamma_2x_2\nu_2.
  \end{array}
  \end{equation}
  Choosing then $\nu=\nu^*:=\begin{bmatrix}
    1 & 2
  \end{bmatrix}^T$, we obtain that
   \begin{equation}
  \begin{array}{lcl}
    LV(x)   &=&     k_1-\gamma_1x_1-2\gamma_2x_2\\
                 &\le&  c_1-c_2V(x)
  \end{array}
  \end{equation}
  where $c_1=k_1$ and $c_2=\min\{\gamma_1,\gamma_2\}$. Considering now the inequality \eqref{eq:dd2} with $\nu=\nu^*$, we get that
  \begin{equation}
    \begin{array}{lcl}
      LV(x)^2-(LV(x))^2&=&k_1+\gamma_1 x_1+4\gamma_2x_2\\
                                        &\le&c_3+c_4 V(x)
    \end{array}
  \end{equation}
  where $c_3=k_1$ and $c_4=\max\{\gamma_1,2\gamma_2\}$. Hence, by virtue of Theorem \ref{th:ergodic}, the conclusion follows.
%
%

\subsection{Proof of Proposition \ref{prop:mu}}\label{ap:mu}

The question that has to be answered is whether for any $\mu$ the set of equations
\begin{equation}\label{eq:dksldksld}
  \begin{array}{rcl}
     k_1+(b-\gamma_1)x_1^*-bx_1^{*2}-bv^*&=&0\\
    -\dfrac{b}{2}x_1^*-\gamma_2\mu+\dfrac{b}{2}x_1^{*2}+\dfrac{b}{2}v^*&=&0
  \end{array}
\end{equation}
has a solution in terms of $k_1$ and $x_1^*$, where $x_1^*$ and $v^*$ are equilibrium values for $x_1$ and $v$. In the following, we define $S(t):=x_1(t)^{2}-x_1(t)+v(t)$ and let $S^*=S^*(k_1)$ be its value at equilibrium that satisfies the first equation of the system \eqref{eq:dksldksld}. In this respect, the above equations can be rewritten as
\begin{equation}\label{eq:dskldjsldsld5678}
  \begin{array}{rcl}
     k_1-\gamma_1x_1^*-bS^*&=&0\\
     -\gamma_2\mu+\dfrac{b}{2}S^*&=&0.
  \end{array}
\end{equation}
Based on the above reformulation, we can clearly see that if we can set $S^*$ to any value by a suitable choice of $k_1$, then any $\mu$ can be achieved. We prove this in what follows.

\textbf{Step 1.} First of all, we have to show that when $k_1=0$, we have that $S^*=0$ and $x_1^*=0$. This can be viewed directly from the results of \cite{Briat:13g} which states that the asymptotic moment bounds for the first-order moment of $V(x)=\nu^T x$ is given by $c_1/c_2$, i.e. $\lim_{t\to\infty}\E[V(X(t))]\le c_1/c_2$, where $c_1,c_2$ are defined in Theorem \ref{th:ergodic}. Choosing $\nu=\begin{bmatrix}
  1 & 0
\end{bmatrix}^T$, suitable $c_1$ and $c_2$ are given by $c_1=k_1$ and $c_2=\gamma_1$. Therefore, $\lim_{t\to\infty}\E[X_1(t)]\le c_1/c_2$. This implies that when $k_1=0$, then $\E[X_1(t)]\to x_1^*=0$ as $t\to\infty$.


\textbf{Step 2.} We show now that when $k_1$ grows unbounded, then $S^*$ grows unbounded as well. To do so, let us focus on the first equation of \eqref{eq:dskldjsldsld5678}. Two options: either both $x_1^*$ and $S^*$ tend to infinity, or only one of them grows unbounded and the other remains bounded. We show that $S^*$ has to grow unbounded. Let us assume that $S^*=S^*(k_1)$ is uniformly bounded in $k_1$, i.e. there exists $\bar{S}>0$ such that $S^*\in[0,\bar{S}]$ for all $k_1\ge0$. Then, from the first equation of \eqref{eq:dskldjsldsld5678}, we have that $x_1^*=(k_1-bS^*)/\gamma_1$ and thus  $x_1^*\ge\bar{x}_1:=(k_1-b\bar{S})/\gamma_1$ for all $k_1\ge0$. Hence, $x_1^*$ grows unbounded as $k_1$ increases to infinity. From Jensen's inequality, we have that $S^*\ge x_1^*(x_1^*-1)$. Noting then that for the function $f(x):=x(x-1)$, we have that $f(y)\ge f(x)$ for all $y\ge x$, $x\ge1$, we can state that
\begin{equation}
  \bar{x}_1(\bar{x}_1-1)\le x_1^*(x_1^*-1)\le S^*\le\bar{S}
\end{equation}
for all $k_1>0$ such that $\bar{x}_1\ge1$. It is now clear that for any $\bar{S}>0$, there exists $k_1>0$ such that the above inequality is violated since $f(\bar{x}_1)$ can be made arbitrarily large. Therefore, $S^*$ must go to infinity as $k_1$ goes to infinity.

Using finally the continuity assumption of the function $S^*(k_1)$, i.e. Assumption \ref{ass:2}, we can conclude that for any $\mu>0$, there will exist $k_1>0$, such that we have $x_2^*=\mu$.  The proof is complete.

\subsection{Proof of Theorem \ref{th:main}}\label{ap:main}


%

\paragraph{Location of the equilibrium points and local stability results}

We know from Proposition \ref{prop:mu} that for any $\mu>0$, the set of equations \eqref{eq:dksldksld}
 has solutions in terms of the equilibrium values $x_1^*,x_2^*,I^*$ and $v^*$. Adding two times the second equation to the first one and multiplying the second one by $2/b$, we get that
\begin{equation}
  \begin{array}{rcl}
    k_cI^*-\gamma_1 x_1^*-\gamma_2\mu&=&0\\
    x_1^{*2}-x_1^*+v^*-\dfrac{2\gamma_2\mu}{b}&=&0.
  \end{array}
\end{equation}
The first equation immediately leads to $ I^*=(\gamma_1 x_1^*+\gamma_2\mu)/k_c$
which is positive for all $\mu>0$. This also means that $x_1^*$ is completely characterized by the equation
\begin{equation}
    x_1^{*2}-x_1^*+v^*-\dfrac{2\gamma_2\mu}{b}=0.
\end{equation}
The goal now is to determine the location of the solutions $x_1^*$ to the above equation where $v^*$ is viewed as an unknown parameter, reflecting our complete ignorance on the value $v^*$. We therefore embed the actual unique equilibrium point (from ergodicity and moments convergence) in a set having elements parametrized by $v^*\ge0$. The equation to be solved is quadratic, and it is a straightforward implication of Descartes' rule of signs \cite{Khovanskii:91} that we have three distinct cases:
\begin{itemize}
  \item[1)] If $v^*-2\gamma_2\mu/b<0$, then we have one positive equilibrium point.
  \item[2)] If $v^*-2\gamma_2\mu/b=0$, then we have one equilibrium point at zero, and one which is positive.
  \item[3)] If $v^*-2\gamma_2\mu/b>0$, then we have either 2 complex conjugate equilibrium points, or 2 positive equilibrium points.
\end{itemize}
  \textbf{Case 1:} This case holds whenever $v^*\in\left[0,\frac{2\gamma_2\mu}{b}\right)$
and the only positive solution for $x_1^*$ is given by $\textstyle{x_1^*=\frac{1}{2}\left(1+\sqrt{\Delta}\right)}$
where
\begin{equation}\label{eq:delta}
\Delta=1+4\left(\dfrac{2\gamma_2\mu}{b}-v^*\right)>1.
\end{equation}
The equilibrium point is therefore given by
\begin{equation}
  z^*=\left[\dfrac{1}{2}\left(1+\sqrt{\Delta}\right),\mu, \dfrac{\gamma_2\mu+\gamma_1x_1^*}{k_c}\right].
\end{equation}
The linearized system around this equilibrium point reads
\begin{equation}
  \dot{\tilde{z}}(t)=\begin{bmatrix}
    -\gamma_1-b\sqrt{\Delta} & 0 & k_c\\
    b\sqrt{\Delta}/2 & -\gamma_2 & 0\\
    0 & -1 & 0
  \end{bmatrix}\tilde{x}(t)+\begin{bmatrix}
      -b\\
      b/2\\
      0
    \end{bmatrix}\tilde{v}(t)
\end{equation}
where $\tilde{z}(t):=z(t)-z^*$, $z(t):=\col(x(t),I(t))$ and $\tilde{v}(t):=v(t)-v^*$. The Routh-Hurwitz criterion allows us to derive the stability condition
\begin{equation}\label{eq:condKc1}
  0<k_c<\dfrac{2\gamma_2(\gamma_1+\gamma_2+b\sqrt{\Delta})(\gamma_1+b\sqrt{\Delta})}{b\sqrt{\Delta}}.
\end{equation}
  \textbf{Case 2:} In this case, we have $v^*=2\gamma_2\mu/b$
and is rather pathological but should be addressed for completeness. Let us consider first the equilibrium point $x_1^*=0$ giving $z^*=\begin{bmatrix}
    0 & \mu & \gamma_2\mu/k_c
  \end{bmatrix}$.
%
The linearized dynamics of the system around this equilibrium point is given by
\begin{equation}
  \dot{\tilde{z}}(t)=\begin{bmatrix}
    b-\gamma_1 & 0 & k_c\\
    -b/2 & -\gamma_2 & 0\\
    0 & -1 & 0
  \end{bmatrix}\tilde{z}(t)+\begin{bmatrix}
      -b\\
      b/2\\
      0
    \end{bmatrix}\tilde{v}(t).
\end{equation}
Since the determinant of the system matrix is positive, this equilibrium point is unstable. Considering now the equilibrium point $x_1^*=1$ and, thus, $ z^*=\begin{bmatrix}
    1 & \mu & (\gamma_2\mu+\gamma_1)/k_c
  \end{bmatrix}$,
we get the linearized system
\begin{equation}
   \dot{\tilde{z}}(t)=\begin{bmatrix}
    -b-\gamma_1 & 0 & k_c\\
    b/2 & -\gamma_2 & 0\\
    0 & -1 & 0
  \end{bmatrix}\tilde{z}(t)+\begin{bmatrix}
      -b\\
      b/2\\
      0
    \end{bmatrix}\tilde{v}(t).
\end{equation}
which is exponentially stable provided that
\begin{equation}\label{eq:condKc2}
  k_c<\dfrac{2\gamma_2(\gamma_1+\gamma_2+b)(\gamma_1+b)}{b}.
\end{equation}
  \textbf{Case 3:} This case corresponds to when $v^*>2\gamma_2\mu/b$.
%
However, this condition is not sufficient for having positive equilibrium points and we must add the constraint $v^*<2\gamma_2\mu/b+1/4$
in order to have a real solutions for $x_1^*$. When the above conditions hold, the equilibrium points are given by
\begin{equation}
  z_\pm^*=\begin{bmatrix}
    \dfrac{1}{2}\left(1\pm\sqrt{\Delta}\right) & \mu & \dfrac{\gamma_2\mu+\gamma_1x_1^*}{k_c}
  \end{bmatrix}
\end{equation}
where  $\Delta$ is defined in \eqref{eq:delta}.
Similarly to as previously, the equilibrium point $z_-^*$ can be shown to be always unstable and the equilibrium point $z^*_+$ exponentially stable provided that
\begin{equation}\label{eq:condKc3}
  k_c<\dfrac{2\gamma_2(\gamma_1+\gamma_2+b\sqrt{\Delta})(\gamma_1+b\sqrt{\Delta})}{b\sqrt{\Delta}}.
\end{equation}
Note that if the discriminant $\Delta$ was equal to 0, the system would be unstable.

\paragraph{Bounds on the variance}

We have thus shown that to have a unique positive locally exponentially stable equilibrium point, we necessarily have an equilibrium variance $v^*$ within the interval $ v^*\in\left(0,2\gamma_2\mu/b+1/4\right].$
If $v^*$ is greater than the upper-bound of this interval, the system does not admit any real equilibrium points.

\paragraph{Uniform controller bound}

The last part concerns the derivation of a uniform condition on the gain of the controller $k_c$ such that all the positive equilibrium points that can be locally stable are stable. In order words, we want to unify the conditions \eqref{eq:condKc1}, \eqref{eq:condKc2} and \eqref{eq:condKc3} all together. Noting that these conditions can be condensed to $k_c<f(\sqrt{\Delta})$
where
\begin{equation}
  f(\zeta):=\dfrac{2\gamma_2(\gamma_1+\gamma_2+b\zeta)(\gamma_1+b\zeta)}{b\zeta}.
\end{equation}
Moreover, since $\mu>0$ can be arbitrarily large (and thus $\Delta$ may take any nonnegative value), the worst case bound for $k_c$ coincides with the minimum of the above function for $\zeta\ge0$. Standard calculations show that the minimizer is  given by $\textstyle \zeta^*=(\gamma_1(\gamma_1+\gamma_2))^{1/2}/b$ and the minimum $f^*:=f(\zeta^*)$ is therefore given by
\begin{equation}
  f^*=2\gamma_2\left(2\gamma_1+\gamma_2+2\sqrt{\gamma_1(\gamma_1+\gamma_2)}\right).
\end{equation}
The proof is complete.

\bibliographystyle{IEEEtranS}
\bibliography{../../../../Lastbib/global,../../../../Lastbib/briat}

\end{document}